\documentclass[12pt]{amsart}
\usepackage[T1]{fontenc}
\usepackage{amsthm}
\usepackage{amsfonts}
\usepackage{amsmath}
\usepackage{amssymb}
\usepackage{enumitem}
\usepackage{fullpage}
\usepackage[style=alphabetic,sorting=nty, minalphanames=4, maxalphanames=4]{biblatex}

\usepackage{etoolbox}
\usepackage{amsmath}
\usepackage{mathtools}
\usepackage{bbm}
\AtBeginEnvironment{align}{\setcounter{equation}{0}}

\newcommand{\R}{\mathbb{R}}

\newcommand{\N}{\mathbb{N}}

\newcommand{\Hd}{\mathcal{H}}

\DeclareMathOperator{\gr}{Gr}
\DeclareMathOperator{\diam}{diam}

\DeclareMathOperator{\Gr}{Gr}

\newcommand{\e}{\varepsilon}
\newcommand{\de}{\delta}

\newtheoremstyle{lemmastyle}
{1em} 
{1em} 
{\slshape} 
{} 
{\bfseries} 
{.} 
{1em} 
{} 

\usepackage{thmtools}

\declaretheoremstyle[%
  spaceabove=0pt,%
  spacebelow=0pt,%
  headfont=\normalfont\itshape,%
  postheadspace=1em,%
  qed=\qedsymbol%
]{mystyle}

\theoremstyle{definition}
\newtheorem*{definition}{Definition}
\newtheorem*{remark}{Remark}
\newtheorem{theorem}{Theorem}
\newtheorem{proposition}{Proposition}[section]
\newtheorem{lemma}[proposition]{Lemma}

\numberwithin{theorem}{section}
\newtheorem*{theorem*}{Theorem}

\begin{document}
\title{Failure of weak-type endpoint restriction estimates for quadratic manifolds.}
\author{Sam Craig}
\begin{abstract} 

    It is well-known that the Fourier extension operator for the paraboloid in $\R^d$ cannot be weak-type bounded at the restriction endpoint $q = 2d/(d-1)$, since such an estimate would imply bounds for the Kakeya maximal function which contradict the existence of Besicovitch sets. We generalize this approach to prove that the Fourier extension operator for an $n$-dimensional quadratic manifold $\mathcal{M}$ cannot be weak-type bounded at the restriction endpoint. The key step in this proof is constructing a set $K \subset \R^d$ containing a translate of every plane normal to $\mathcal{M}$ which can be covered by $\lesssim \de^{-d}\left(\frac{\log \log (1/\de)}{\log (1/\de)}\right)^{n/(d-n)}$ many $\de$-balls. Such a set rules out endpoint bounds for the associated Kakeya maximal function and hence weak-type endpoint estimates for the restriction operator. 
\end{abstract}
\maketitle

    \section{Introduction}
    
    Let $F_1(\xi), \dots, F_{d-n}(\xi)$ be quadratic homogeneous polynomials on $\R^n$. Their graphs over an open ball $\Omega \subset \R^n$ define a quadratic manifold $\mathcal{M} \subset \R^d$. For such manifolds, we can define the \textit{Fourier extension operator} $\mathcal{E}^{\mathcal{M}}$ mapping Schwartz functions on $\Omega$ to Schwartz functions on $\R^{d}$ by \[\mathcal{E}^{\mathcal{M}}f(x) = \int_{\Omega} e^{2 \pi i x\cdot (\xi, F_1(\xi), \dots, F_{d-n}(\xi))} f(\xi)~d\xi.\]The Fourier restriction problem for $\mathcal{M}$ is to determine the range of $p$ and $q$ for which \[||\mathcal{E}^{\mathcal{M}}f||_{L^q(\R^{d})} \lesssim ||f||_{L^p(\Omega)}.\]
    By applying the extension operator to bump functions supported on small balls in $\Omega$, it is easy to see that $\mathcal{E}^{\mathcal{M}}$ cannot be bounded if $q \le \frac{2d}{n}$. However, this example does not rule out a weak-type bound for the extension operator at $q = \frac{2d}{n}$. The main result in this paper is that the extension operator does not satisfy restricted weak-type bounds at $q = \frac{2d}{n}$. 
    \begin{theorem}\label{thm:main}
        If $\mathcal{M}$ is a $n$-dimensional quadratic manifold in $\R^{d}$, the operator $\mathcal{E}^{\mathcal{M}}$, initially defined on Schwartz functions, does not extend to a restricted weak-type $(p, 2d/n)$ bound. 
    \end{theorem}
    
    This generalizes a 1998 paper by Beckner, Carbery, Semmes, and Soria. They prove that if the extension operator for the paraboloid is bounded $L^{p} \to L^{2d/(d-1), \infty}$, then standard Kakeya sets must have strictly positive measure, contradicting the existence of measure zero Kakeya sets~\cite{BCSS89}. We slightly modify their approach by proving that restricted weak-type endpoint bounds for $\mathcal{E}^{\mathcal{M}}$ imply restricted weak-type bounds for an associated generalization of the Kakeya maximal function, then proving that such bounds cannot hold. In particular, we define the Kakeya maximal function $K^{\mathcal{N}}_{\de}:L^1_{\text{loc}}(\R^d) \to L^1(\mathcal{N}, \mathcal{H}_{\Gr(d-n,d)}^n)$ associated with $\mathcal{N}$ by \[K^{\mathcal{N}}_{\de}(f)(P) = \sup_{a \in \R^d} \de^{-n} \int_{P_{\de}(a)} |f|(x)~dx,\]where $P_{\de}(a)$ is the intersection of the $\de$-neighborhood of $P$ with the unit ball, translated by $a$. We say $K^{\mathcal{N}}_{\de}$ satisfies restricted weak-type $(p,q)$ bounds with constant $C$ if it is bounded $L^{p,1}(\R^d) \to L^{q, \infty}(\mathcal{N}, \mathcal{H}_{\Gr(d-n,d)}^n)$ with constant $C$. We prove that $K^{\mathcal{N}}_{\de}$ does not satisfy restricted weak-type $(p,q)$ bounds with constant independent of $\de$.
    
    \begin{theorem}\label{thm:kmf}
        For any $p \in [1, \infty)$, any $q \in [1, \infty]$, and any $\de \in (0, 1)$, the Kakeya maximal function $K_{\de}$ associated with the $\mathcal{N} \subset \text{Gr}(d, d-n)$ of normal planes to $\mathcal{M}$ is does not satisfy restricted weak-type $(p,q)$ bounds with constant $\ll \left(\frac{\log 1/\de}{\log \log  1/\de}\right)^{n/(p(d-n))}$.

    \end{theorem}

    We also prove a novel existence result for measure zero generalized Kakeya sets coming from families of planes. Specifically, for $\mathcal{N} \subset \Gr(d-n, d)$, we define an $\mathcal{N}$-Kakeya set to be a compact set $E$ such that for each $P \in \mathcal{N}$, there exists $a \in \R^d$ and a unit ball $B \subset P$ such that $a + B \subset E$. We prove that if $\dim(\mathcal{N}) = n$, then there exist $\mathcal{N}$-Kakeya sets of Lebesgue measure $0$, and that the same holds for lower-dimensional $\mathcal{N}$-Kakeya sets with lower dimensional Hausdorff measures taking the place of the Lebesgue measure.

    \begin{theorem}\label{thm:hkak}
        For any $\mathcal{N} \subset \gr(d-n,d)$ with $\dim_H(\mathcal{N}) := s \in (0,n]$, there exists an $\mathcal{N}$-Kakeya set $E$ with $\mathcal{H}^{s + d-n}(E) = 0$, where $\mathcal{H}^{s+d-n}$ denotes the $s+d-n$-dimensional Hausdorff measure.
    \end{theorem}

    \section{Background and proof outline}

    Restriction estimates for quadratic manifolds are most commonly studied in the codimension one case, in particular when $\mathcal{M}$ is a paraboloid or sphere. While the higher codimension case is also an area of active research, it is less well-understood. The current best restriction estimates for quadratic manifolds appear in a recent paper of Gan, Guth, and Oh~\cite{GGO23}. These estimates are complicated to state in general, but for certain $n$-dimensional quadratic manifolds of codimension two, they prove $L^p \to L^p$ restriction estimates when $p \ge  \frac{2(n+2) + 2}{n} + O(n^{-2}).$ It is known that many quadratic manifolds will have restriction endpoints above the dimensionally optimal restriction endpoint $\frac{2d}{n}$. In the case codimension one or two, there are conjectures classifying which quadratic manifolds should satisfy the dimensionally optimal restriction estimates,

    While we can rule out weak-type endpoint results for restriction to quadratic manifolds, the same cannot be said for restriction operators in general. In particular, a 2009 paper by Bak, Oberlin, and Seeger proves endpoint restriction estimates for the moment curve in $\ge 3$ dimensions \cite{BOS09}. Outside of certain curves, hypersurfaces, and quadratic manifolds, little is known about when weak-type endpoint restriction estimates may hold or fail. 

    As in the case of restriction, Kakeya maximal functions are most commonly studied in the codimension one case corresponding to families of line segments with directions in $S^{d-1}$, in which case they imply lower bounds for the dimension of Kakeya sets in $\R^d$. The generalized Kakeya maximal functions defined above have also been studied in the case where $\mathcal{N} = \text{Gr}(d-n, d)$, in which case they imply lower bounds for the measures of sets in $\R^d$ containing a translated $d-n$ plane in every direction \cite{O05, B91}. 

    Most standard Kakeya set constructions come from constructing a Kakeya set in $\R^2$, for example using the Perron tree construction, then taking products to arrive at higher dimensional analogs. For certain choices of direction set $\mathcal{N}$, such an approach yields measure zero $\mathcal{N}$-Kakeya sets, but this cannot apply for general choices of $\mathcal{N}$. In 1999, Keich reviewed several techniques to construct standard Kakeya sets, including the Perron tree construction \cite{K99}. Our approach for Theorems \ref{thm:kmf} and \ref{thm:hkak} was inspired by a 2003 paper of K\"orner constructing Kakeya sets in $\R^2$ using the Baire category theorem \cite{K03}. 

    \subsection*{Proof outline}

    In Section \ref{sec:kmf}, we prove Theorem \ref{thm:kmf}. The key step in this proof is constructing an $\mathcal{N}$-Kakeya set in $[-2,2]^d$ which can be covered by $\lesssim \de^{-d}\left(\frac{\log \log 1/\de}{\log 1/\de}\right)^{n/(d-n)}$ many $\de$-balls. To do so, we assume $\mathcal{N}$ is contained in a small ball of fixed size around the plane spanned by the first $d-n$ coordinate directions. We prove that for any $\mathcal{N}$-Kakeya set $K$ and any axis-parallel rectangle $A$ short in the first $d-n$ directions and long in the remaining $n$ directions, we can perturb $K$ slightly to find another $\mathcal{N}$-Kakeya set $K_1$ with small measure in $A$. Repeating this process for a cover of $[-2,2]^d$ consisting of $m$ different choices of $A$ results in an $\mathcal{N}$-Kakeya $K_m$ which is a small perturbation from being small in each rectangle $A$. With some book-keeping, we conclude that $K_m$ can be covered by $\de^{-d}\left(\frac{\log \log 1/\de}{\log 1/\de}\right)^{n/(d-n)}$ many $\de$-balls. We make this argument rigorous in the proof of Proposition \ref{prop:propKak}. The proof of Theorem \ref{thm:kmf} follows easily from this construction.

    In Section \ref{sec:main}, we deduce Theorem \ref{thm:main} from Theorem \ref{thm:kmf}. We use standard reductions to prove that bounds for the restriction operator associated to a manifold $\mathcal{M}$ imply bounds for the Kakeya maximal function $\mathcal{K}^{\mathcal{N}}_{\de}f$ associated with the set $\mathcal{N} \subset \text{Gr}(d-n,d)$ of normal planes to $\mathcal{M}$.

    In Section \ref{sec:hkak}, we prove Theorem \ref{thm:hkak} using topological methods similar to those used by K\"orner for standard Kakeya sets. We equip the collection $\mathcal{P}$ of all $\mathcal{N}$-Kakeya sets with the Hausdorff metric, forming a complete metric space. We define a countable collection of open dense sets in $\mathcal{P}$ whose intersection consists only of measure zero sets, then use the Baire category theorem to complete the proof.

    \subsection*{Acknowledgements}
    The author thanks his advisor, Betsy Stovall, for suggesting this problem and for her advice and support. The author also thanks Marianna Csörnyei for introducing him to K\"orner's paper. The author was supported during this project by the National Science Foundation under grant numbers DMS-2037851 and DMS-2246906. The author thank the anonymous reviewer for their suggestions on the manuscript.

    \subsection*{Notation}
    \begin{itemize}
        \item In the remainder of the paper, the choice of manifold $\mathcal{M}$ is fixed, so we denote $\mathcal{E}^{M}$ by $\mathcal{E}$.
        \item If there exists a constant $C$ depending only on $d, \mathcal{M}$ such that $A \le CB$, we write $A \lesssim B$. If the constant also depends on a collection of parameters $p_1, \dots, p_m$, we write $A \lesssim_{p_1, \dots, p_m} B$. If $A \lesssim B$ and $B \lesssim A$, then we write $A \approx B$. We may treat approximate equality as exact equality when doing so does not introduce confusion, in particular for the dimensions of plates.
        \item We denote the Grassmannian of $n$-planes in $\R^d$ by $\gr(n,d)$.
        \item We define $F: \R^n \to \R^d$ to be the function $F(\xi) = (\xi, F_1(\xi), \dots, F_{d-n}(\xi))$.
        \item For an open rectangle or ball $\square \subset \R^n$, we use $\chi_{\square}$ to denote a smooth bump function adapted to $\square$, $\mathbf{1}_{\square}$ to denote the indicator function for $\square$, and $\Gamma_{\de}(\square)$ to denote a $\de$-net of $\square$, that is, a maximal $\de$-seperated set in $\square$.
        \item We denote the ball centered at $x$ with radius $r$ by $B(x,r)$, the $d$-dimensional unit sphere $\{x \in \R^{d+1} : |x| = 1\}$ by $S^d$, and the standard basis for $\R^n$ by $e_1, \dots, e_n$.
        \item We denote by $d_H$ the Hausdorff metric on compact subsets of $\R^d$ and by $d_{\Gr}$ the metric on the Grasmannian given by $d_{\Gr}(V,W) = d_H(V \cap S^{d-1}, W \cap S^{d-1})$. 
    \end{itemize}

    \section{Necessary conditions for bounds on the Kakeya maximal function}\label{sec:kmf}

    We begin our proof of Theorem ~\ref{thm:kmf} with the construction of a set $K$ which can be covered by few $\de$-balls while still containing a plane segment in every direction of $\mathcal{N}$. To do so, we consider $\mathcal{N}$-Kakeya sets which are a finite union of right rectangular prisms with $n$ short sides of length $\de$ and $d-n$ long sides of length $1$. Going forward, we call such prisms $\de$-prisms and call the $\mathcal{N}$-Kakeya set which comprise the prisms a \emph{scale $\de$ $\mathcal{N}$-Kakeya set}. We denote the number of $\rho$-balls necessary to cover a set $K$ by $|K|_{\rho}$. 

    \begin{proposition}\label{prop:propKak}
        For any $\de > 0$ and any $\mathcal{N} \subset \Gr(d-n,d)$ satisfying $|\mathcal{N}|_{\rho} \lesssim \rho^{-n}$ for all $\rho \in [\de,1)$, there exists a scale $\de$ $\mathcal{N}$-Kakeya set $K$ satisfying \[|K|_{\de} \lesssim \de^{-d}\left(\frac{\log \log 1/\de}{\log 1/\de}\right)^{n/(d-n)}.\]
    \end{proposition}

    We may restrict the choices of $\mathcal{N}$ to those where $d_{\Gr}(\mathcal{N}, \text{span}\{e_1, \dots, e_{d-n}\}) \le c_{d,n}$ for a dimensional constant $c_{d,n}$ in our proof of Proposition \ref{prop:propKak}, since we can then acheive the general case by rotating and taking finite unions. We begin with a simple geometric lemma. Denote by $d'_H(A,B) = \sup_{x \in B} \inf_{y \in A} |x-y|$ the asymmetric distance between $A$ and $B$.
    \begin{lemma}\label{lem:balls}
        If $d'_H(A,B) < \rho$, then $|B|_{\de + \rho} \le |A|_{\de}$. 
    \end{lemma}
    \begin{proof}
        This is immediate if $|A|_{\de}$ is infinite. Otherwise, if $\{B(x_i, \de)\}$ is a collection of $|A|_{\de}$ many balls covering $A$, then $\{B(x_i, \de + \rho)\}$ is the same number of $\de + \rho$ balls covering $B$, since each point in $B$ is within $\rho$ of a point in $A$. 
    \end{proof}
    We now prove that for any scale $\de$ $\mathcal{N}$-Kakeya set $K$ and any $\e > 0$, for any "horizontal strip" $A_{h,\e} = \left(\prod_{i=1}^{d-n} [h_i-\e, h_i+\e]\right) \times [-2,2]^n$ we can find an $\mathcal{N}$-Kakeya set $K'$ at scale $\de \e$ which is small in $A_{h,\e}$ and is close with respect to $d'_H$ to $K$. 
    \begin{lemma}\label{lem:induct}
        For any $h \in [-2,2]^{d-n}$ and $\de, \e > 0$, if $K$ is a scale $\de$ $\mathcal{N}$-Kakeya set and $\rho = \de \e$, then there exists a scale $\rho$ $\mathcal{N}$-Kakeya $K'$ such that $|K' \cap A_{h,\e}|_{\rho} \le \de^{n-d}|\mathcal{N}|_{\de}$ and $d'_H(K', K) < \de$. 
    \end{lemma}

    \begin{proof}
        Take a collection of $m \approx |\mathcal{N}|_{\de}$ $\de$-balls $B(P_1, \de), \dots, B(P_m, \de)$ covering $\mathcal{N}$. We know $K$ contains a plane segment in the direction of $P_i$ translated by some $\tau_i \in A_{h, \e}$ for each $i$. For each $i$, the cone consisting of $\rho$ prisms of the directions in $B(P_i, \de)$, centered at $\tau_i$ has measure $\e^{d-n} \times (\de \e)^n$ in $A_{h, \e}$. The union $K'$ of these prisms is a scale $\rho$ $\mathcal{N}$-Kakeya set with $d'_H(K',K) < \de$ and total measure in $A_{h,\e}$ of $\approx \e^d \de^{d-n} |\mathcal{N}|_{\de}$ and hence $|K' \cap A_{h,\e}|_{\rho} \lesssim \de^{n-d}|\mathcal{N}|_{\de}$
    \end{proof}

    Finally, we prove the proposition. We first outline the proof. Fix a parameter $\e > 0$ and take $k = \e^{n-d}$ the $\e$-covering number of $[-2,2]^{d-n}$. Eventually, we take $\de = \e^{k+1}$. Eumerate an $\e$-covering set of $[-2, 2]^{d-n}$ as $h_0, h_1, \dots, h_k$ and let $A_i = A_{h_i, \e}$. We start with a scale $1$ cone $K_0$ centered at the first horizontal stripe $A_{h_0, \e}$, which we denote $A_0$. This has small measure in $A_0$ and is an $\mathcal{N}$-Kakeya set. We can use Lemma \ref{lem:induct} to find an $\mathcal{N}$-Kakeya set $K_1$ which is small in another horizontal strip $A_1$ and is close with respect to $d'_H$ to $K_0$. We repeat this process until we have used up all of the strips and arrive a scale $\e^{k+1}$ $\mathcal{N}$-Kakeya set $K_k$. At this point, we use Lemma \ref{lem:balls} to see that $K_k$ is approximately the same size as $K_j$ in $A_j$. Adding up the contribution from each strip, we arrive at the desired size for $|K_k|_{\de}$. 
    \begin{proof}[Proof of Proposition \ref{prop:propKak}]
        Fix $\de > 0$ and choose $\e$ to satisfy $\e^{\e^{-(d-n)} + 1} = \de$. Since the mapping $\e \mapsto \e^{\e^{-(d-n)} + 1}$ is a decreasing continuous map $(0,1] \to (0,1]$, so we can find such an $\e \in (0, 1]$ for any $\de \in (0, 1]$ . Take $k = \e^{-(d-n)}$ and let $\rho_j = \e^j$ for $j = 0, 1, \dots, k$. Note that $\rho_j > \de$ for all $j$. We can find a scale $\rho_0$ $\mathcal{N}$-Kakeya set $K_0$ with $|K_0 \cap A_0|_{\rho_1} = 1 \approx \rho_0^{n-d} |\mathcal{N}|_{\rho_0}$ in $A_0$ by taking the cones in the previous lemma all centered at a common point in $A_0$. Since we've assumed that all the planes in $\mathcal{N}$ are quantitatively transverse to the long directions of $A_0$, these cones intersect $A_0$ in an $\e$ ball. Obtain $K_{i+1}$ from $K_i$ by applying Lemma \ref{lem:induct} with $h:=h_{i+1}$. From the conclusion of the lemma, we know that $|K_{i+1} \cap A_{h_{i+1}}|_{\rho_{i+1}} \le \rho_i^{n-d}|\mathcal{N}|_{\rho_i}$. Applying Lemma \ref{lem:balls} for each $j \le i$, we see that $K_{i+1} \cap A_{h_j}$ can be covered by $\rho_j^{n-d}|\mathcal{N}|_{\rho_j}$ many balls of radius less than $\rho'_{j+1} = \sum_{k = j+1}^{\infty} \rho_k \approx \rho_{j+1}$. Hence, $|K_{i+1} \cap A_{h_j}|_{\rho'_{j+1}} \lesssim {\rho'}_j^{n-d}|\mathcal{N}|_{\rho'_j}$ for each $j \le i+1$. 

        After $k$ steps, we have gone over each horizontal strip and conclude with a set $K := K_{k}$. Choose $\e$ so that $\de = \rho'_{k+1} \approx \e^{k+1} = \e^{1/\e^{d-n} + 1}$. We see that \[|K \cap A_{h_j}|_{\de} \lesssim \left(\frac{\e\rho'_j}{\de}\right)^d {\rho'}_j^{n-d}|\mathcal{N}|_{\rho'_j} = \e^d \de^{-d} {\rho'}_j^{n-d}|\mathcal{N}|_{\rho'_j} \lesssim \e^d \de^{-d}.\]Since $j$ takes $\approx \e^{n-d}$ many values, we conclude that we can cover $K$ with $\approx \e^{n}\de^{-d}$ many $\de$ balls. This is $\approx \de^{-d}\left(\frac{\log \log 1/\de}{\log 1/\de}\right)^{n/(d-n)}$ so long as $\left(\frac{\log 1/\de}{\log\log 1/\de}\right)^{{n/(d-n)}} \lesssim \e^{n}$. Since we have taken $\de = \e^{1/\e^{d-n} + 1}$, we see that $\log(1/\de) \approx \e^{n-d}\log(1/\e)$ and $\log \log (1/\de) \approx \log(1/\e)$, so $\left(\frac{\log 1/\de}{\log\log 1/\de}\right)^{{n/(d-n)}} \approx \e^{n}$, as desired.
    \end{proof}
    \begin{proof}[Proof of Theorem ~\ref{thm:kmf}]
        Let $\mathcal{N}$ be the normal space to $\mathcal{M}$. As $\mathcal{M}$ is a compact $n$-dimensional manifold, it satisfies $|\mathcal{M}|_{\de} \lesssim \de^{-n}$. The map sending a point in $\mathcal{M}$ to the normal vector to $\mathcal{M}$ at that point is a Lipschitz map sending $\mathcal{M}$ to $\mathcal{N}$, so we have that $|\mathcal{N}|_{\de} \lesssim \de^{-n}$ as well. Now that we have checked the hypothesis for $\mathcal{N}$, we can apply Proposition \ref{prop:propKak} for any $\de > 0$. Let $K$ be the resulting scale $\de$ $\mathcal{N}$-Kakeya set. Note that $K$ can be covered by $\de^{-d}\left(\frac{\log \log 1/\de}{\log 1/\de}\right)^{n/(d-n)}$ many $\de$ balls and hence has measure $\lesssim \left(\frac{\log \log 1/\de}{\log 1/\de}\right)^{n/(d-n)}$. Furthermore, $\mathcal{K}_{\de}(\mathbf{1}_K)(x) \ge 1$ for all $x \in \mathcal{N}$. It follows that $||\mathcal{K}_{\de}(\mathbf{1}_K)||_{L^{q, \infty}} \approx 1$. On the other hand, \[||\mathbf{1}_K||_{L^{p,1}} = |K|^{1/p} \lesssim \left(\frac{\log \log 1/\de}{\log 1/\de}\right)^{n/(p(d-n))}.\]Sending $\de \to 0$, we see that any bound \[||\mathcal{K}_{\de}(\mathbf{1}_K)||_{L^{q,\infty}} \ll \left(\frac{\log 1/\de}{\log \log  1/\de}\right)^{n/(p(d-n))}||\mathbf{1}_K||_{L^{p,1}}\]is impossible.
    \end{proof}

    \begin{remark}
        The modified Perron tree construction of Schoenberg yields standard Kakeya sets in $\R^d$ which can be covered by $\approx \frac{\de^{-d}}{\log (1/\de)}$ many $\de$-balls. We elaborate on the differences between the construction of Proposition \ref{prop:propKak} and that of Schoenberg to explain the slightly worse bound in Proposition \ref{prop:propKak}. 
        
        We first discuss both constructions in the case of Kakeya sets in $\R^2$. Both constructions follow from an iterative process, beginning with an equilateral triangle with altitude of length $1$ on the $y$-axis and repeatedly subdividing into smaller triangles and either rotating (in Proposition \ref{prop:propKak}) or translating (in Schoenberg's construction) the smaller triangles to make the union smaller in a specified horizontal strip. By iterating over enough horizontal strips to cover the entire triangle, we arrive at a Kakeya set with small measure in each horizontal strip and hence small measure in its totality. In Proposition \ref{prop:propKak}, we see that the rotations can be chosen to not significantly increase the measure of the union in the horizontal strips from previous iterations, while in Schoenberg's construction, the translations can be chosen to in fact make the union smaller in both the current horizontal strip and horizontal strips from all previous iterations. This gives a Kakeya set that can be efficiently covered by larger balls.

        In the case of general choices of $\mathcal{N}$, it is not clear how to carry out an analog of Schoenberg's construction. At each step of the iteration process, we would need to translate the "small triangles" (which are now actually cones) in the same direction to make the set smaller in the "horizontal strips" used in the previous iterations. On the other hand, those translations might not actually decrease the measure of the small triangles in the horizontal strip for the current step. Without further assumptions on $\mathcal{N}$, it seems difficult to imitate the approach of Schoenberg.

        We also have assumed that $|\mathcal{N}|_{\de} \approx \de^{-n}$. It is easy to see that this can be weakened to $|\mathcal{N}|_{\de} \lesssim \de^{-n}$, that is, the upper Minkowski dimension of $\mathcal{N}$ is at most $n$. This is true for the normal space to any compact $n$-dimensional smooth manifold and many other choices of $\mathcal{N}$, but does not hold for a general set of Hausdorff dimension $n$. Theorem \ref{thm:hkak} applies for such sets, but is less quantitative, with no explicit upper bound for the $\de$ covering number of it's $\mathcal{N}$-Kakeya sets. 
    \end{remark}
    
    \section{Failure of restricted weak-type bounds for $\mathcal{E}$}\label{sec:main}

        In this section, we reduce Theorem \ref{thm:main} to Theorem \ref{thm:kmf}. This reduction is substantially the same as the well-known proof that bounds on the Fourier extension operator for the paraboloid imply bounds on the standard Kakeya maximal function \cite[Theorem 22.12]{M15}. 
        
        Recall we have defined $F(\xi) = (\xi, F_1(\xi), \dots, F_{d-n}(\xi))$ and $\mathcal{M} = F(\Omega)$ for some open ball $\Omega$. The Fourier extension operator for $\mathcal{M}$ is \[\mathcal{E}f(x) = \int_{\Omega} e^{2 \pi i x\cdot F(\xi)} f(\xi)~d\xi.\]We begin by proving that we can approximate the extension operator applied to a bump function $\chi_{\de, \eta}$ on a $\de$-ball centered at $\eta \in \R^n$ from below with the indicator function of a rectangular prism $T_{\eta, \de}$ centered at $0$ of length $\approx \de^{-1}$ in the $n$ tangent directions to $\mathcal{M}$ at $F(\eta)$ and length $\approx \de^{-2}$ in the $d-n$ normal directions, and that we can approximate the modulated bump function $\chi_{\de, \eta, v}(\xi) := e^{-2 \pi i v \cdot F(\xi)} \chi_{\de, \eta}(\xi)$ from below by the translated prism $T_{\eta, \de, v} := T_{\eta, \de} + v$. 

        \begin{lemma}\label{lem:lem3}
            For $\de \in (0, 1]$, for any $\eta \in \square, v \in \R^{d}$, there exists a large plate $T_{\eta, \de, v}$ of length $\approx \de^{-2}$ in the $d-n$ direction normal to $\mathcal{M}$ at $\eta$ and $\approx \de^{-1}$ in the $n$ directions tangent to $\mathcal{M}$ at $F(\eta)$, with $|\mathcal{E}(\chi_{\eta, \de, v})|(x) \gtrsim \de^n$ for $x \in T_{\eta, \de, v}$.
        \end{lemma}

        First, we define the plates using symmetries of the extension operator. 

        \begin{proposition}\label{prop:sym}
            Let $f:\R^m \to \R$ be an arbitrary function. For an $m\times m$ matrix $M$, denote $x \mapsto f(Mx)$ by $f_M$. If $Mx = \de x$ for some scalar $\de$, we write $f_{\de}$ instead of $f_M$. For $\eta \in \R^m$, denote $x \mapsto f(x - \eta)$ by $f^{\eta}$. has the following symmetries:
            \begin{enumerate}[label=(\alph*)]
                \item $\mathcal{E}(f_{1/\de})(x) = \de^{n}\mathcal{E}(f)(D_{\de}x)$, where $D_{\alpha}$ is the diagonal matrix with $\alpha$ in the first $n$ diagonal entries and $\alpha^2$ on final $d-n$ entries.
                \item $\mathcal{E}(f^{\eta}) = e^{2\pi i x \cdot F(\eta)}\mathcal{E}(f)(A^*(\eta)x)$, where $A(\eta)$ is the lower triangular matrix with ones on the diagonal and entries polynomial in $\eta$ satisfying $F(\xi + \eta) = F(\eta) + A(\eta)F(\xi)$.
                \item $\mathcal{E}(e^{-2 \pi i v \cdot F(\xi)}f)(x) = \mathcal{E}(f)(x-v)$.
            \end{enumerate}
        \end{proposition}

        These symmetries are well-known and follow easily from changes of variables, so we do not prove them here. 

        \begin{proof}[Proof of Lemma~\ref{lem:lem3}]
            We may assume, without loss of generality, that $\Omega = B(0,1)$. For $x$ in a small neighborhood $[-\rho, \rho]^n$ of $0$, $x \cdot F(\xi) \le 1/8$ for all $\xi \in B(0,1)$ and hence if we set $\chi_{0,1} = \chi_{B(0,1)}$, then for $x \in T_{0,1} := [-\rho, \rho]^n$, $\mathcal{E}(\chi_{0,1})(x) \gtrsim 1$. Now set $T_{0,\de} = D_{1/\de}T_{0,1}$ and $\chi_{0,R} = (\chi_{0,1})_R$. Because of the form of the dilation, this plate has the right dimensions. By the first symmetry in Proposition~\ref{prop:sym}, we have that for $x \in T_{0,R}$, $\mathcal{E}(\chi_{0,R})(x) = R^{-n}\mathcal{E}(\chi_{0,1})(D_{1/R}x) \gtrsim R^{-n}$. 

            Next, we define the frequency translated plates. For $\eta \in \square$, set $\tilde{T}_{\eta,R} = (A^*(\eta))^{-1}T_{0,R}$. The transformation $A^*(\eta)^{-1}$ sends the long directions of $T_{0,R}$ to span $\mathcal{N}_{\eta}$, as desired. However, since $(A^*(\eta))^{-1}$ is not an orthogonal matrix, $\tilde{T}_{\eta,R}$ is an oblique rectangular prism, while by our definition, a plate must be a right rectangular prism. Let $V$ denote the plane segment formed by the long directions of $\mathcal{N}_{\eta}$ and let $T_{\eta, R}$ be the $R$ neighborhood of $V$. By rescaling $T_{\eta,R}$ by the $\approx 1$ factor, we have $T_{\eta,R} \subset \tilde{T}_{\eta,R}$. Set $\chi_{\eta,R} = \chi_{0,R}^{\eta}$. Then by the translational symmetry of the extension operator, for $x \in T_{\eta,R}$, $|\mathcal{E}(\chi_{\eta,R})(x)| = |e^{2\pi i x \cdot F(\eta)}\mathcal{E}(f)(A^*(\eta)x)| \gtrsim R^{-n}.$

            Finally, we extend this to spatially translated plates. For $v \in \R^{d}$, set $T_{\eta, R, v} = T_{\eta,R}+ v$. These are clearly plates of the right dimensions and if we set $\chi_{\eta, R, v} = e^{-2 \pi i v \cdot F(\xi)}\chi_{\eta, R}$, then $\mathcal{E}(\chi_{\eta,R,v})(x) = \mathcal{E}(\chi_{\eta,R})(x-v)$. Both desired conditions follow immediately.
        \end{proof}

        Using this approximation, we prove that a restricted weak-type $(p, 2d/n)$ bound for the Fourier extension operator for $\mathcal{M}$ implies a restricted weak-type $(d/(d-n), p/(p-2))$ bound independent of $\de$ for the Kakeya maximal function associated with the normal set to $\mathcal{M}$. We in fact prove a general deduction from any endpoint Lorentz space bounds for the restriction operator.

        \begin{proposition}\label{prop:prop3}
            If $\mathcal{E}$ is bounded $L^{p,s} \to L^{2d/n, r}$ with constant $C$, then for all $\de > 0$, the associated Kakeya maximal function $\mathcal{K}_{\de}$ is bounded from $L^{d/(d-n),r'}$ to $L^{q, s'}$ with constant $\approx C^2$, where $q = \frac{p}{p-2}$ is the H\"older conjugate to $p/2$. 
        \end{proposition}

        \begin{proof}
            Take $f \in L^q(\R^d)$ and $P_1, \dots, P_m$ a maximal $\de$-separated subset of $\mathcal{N}$. Since $\mathcal{N} \subset \bigcup_{i=1}^m B(P_i, \de)$, by a covering argument we have that $||K_{\de}(f)||_{L^{q,s'}} \le ||\de^{n/q}K_{\de}f(P_i)||_{\ell^{q, s'}}$, where $||\cdot||_{\ell^{p,s}}$ denotes the Lorentz quasinorm on $\R^m$ with respect to the counting measure.It then suffices to prove that $||\de^{n/q}K_{\de}f(P_i)||_{\ell^{q, s'}} \le C||f||_{L^{d/(d-n), r'}}$, which by duality would follow from proving $\sum_{i=1}^m \de^{n/q} c_i K_{\de}(f)(P_i) \le C||f||_{L^{d/(d-n), r'}}$ for any sequence $c = (c_1, \dots, c_m)$ with $||c||_{\ell^{p/2, s'}(\R^m)} = 1$. For some choice of translations $v_i \in \R^d$, \[\de^{n/q}K_{\de}(f)(P_i) \approx \de^{-n/q'}\int_{\R^d} \mathbf{1}_{P_{i, \de}(v_i)} |f|(x)~dx,\]so $\sum_{i=1}^m \de^{n/q} c_i K_{\de}(f)(P_i) \le \int_{\R^d}\sum_{i=1}^m\de^{-n/q'}c_i\mathbf{1}_{P_{i, \de}(v_i)} |f|(x)~dx$. By H\"older's inequality, \[\int_{\R^d}\sum_{i=1}^m\de^{-n/q'}c_i\mathbf{1}_{P_{i, \de}(v_i)} |f|(x)~dx \le \left|\left| \sum_{i=1}^m \de^{-n/q'} c_i \mathbf{1}_{P_{i, \de}(v_i)}\right|\right|_{L^{d/n, r}} ||f||_{L^{d/(d-n), r'}}.\] Therefore, it suffices to prove that $\left|\left| \sum_{i=1}^m \de^{-n/q'} c_i \mathbf{1}_{P_{i, \de}(v_i)}\right|\right|_{L^{d/n, r}} \lesssim C^2$, which after rescaling follows from proving that $\left|\left| \sum_{i=1}^m \de^{-n/q'} c_i \mathbf{1}_{T_{i}(v_i)}\right|\right|_{L^{d/n, r}} \le C^2 \de^{-n}$, where $T_i = \de^{-2}P_{i, \de}$.

            Note that there exist $\approx \de$ separated points $\xi_1, \dots, \xi_n$ such that the normal to $\mathcal{M}$ at $F(\xi_i)$ is $P_i$. By Lemma \ref{lem:lem3}, we know $|\mathcal{E}(\de^{-n/(2q')} \sqrt{c_i} \chi_{\xi_i, \de, v_i})|(x) \gtrsim \de^{n-n/(2q')}\sqrt{c_i}\mathbf{1}_{T_i(v_i)}$ for $x \in T_i(v_i)$. It follows that \[\sum_{i=1}^m \de^{-n/q'}c_i\mathbf{1}_{T_i(v_i)}(x) \lesssim \de^{-n}\sum_{i=1}^m |\mathcal{E}(\de^{-n/(2q')} \sqrt{c_i} \chi_{\xi_i, \de, v_i})|^2(x).\]Since $\xi_i$ is an $\approx \de$-separated family and $\chi_{\xi_i, \de, v_i}$ is supported on $B(\xi_i, \de)$, the supports of $\xi_i$ are boundedly overlapping. Then by Khinchine's inequality and our assumed bound for the restriction operator \begin{align*}
                \left|\left| \sum_{i=1}^m \de^{-n/q'} c_i \mathbf{1}_{T_i(v_i)}\right|\right|_{L^{d/n, r}} &\lesssim  \de^{-n}\left|\left|\left(\sum_{i=1}^m |\mathcal{E}(\de^{-n/(2q')} \sqrt{c_i}\chi_{\xi_i, \de, v_i})|^2\right)^{1/2}\right|\right|_{L^{2d/n, r}}^2\\ &\lesssim \de^{-n}C^2||\sum_{i=1}^m \de^{-n/(2q')} \sqrt{c_i}\chi_{\xi_i, \de, v_i}||^2_{L^{p,s}}\\ &\approx \de^{-n}C^2\de^{-n/q'} \left|\left|\sum_{i=1}^m c_i\right|\right|_{\ell^{p/2, s'}}\de^{2n/p} = \de^{-n}C^2
            \end{align*}
        \end{proof}
        
    Theorem \ref{thm:main} follows easily from here, since by Proposition \ref{prop:prop3} any bound $\mathcal{E}:L^{p, 1} \to L^{2d/n, \infty}$ implies a bound $\mathcal{K}_{\de}:L^{d/(d-n),1} \to L^{q, \infty}$ with constant independent of $\de$, contradicting Theorem \ref{thm:kmf}. 

    \section{Constructing measure zero $\mathcal{N}$-Kakeya sets}\label{sec:hkak}
    We now prove Theorem~\ref{thm:hkak}. This proof generalizes that of K\"orner, which covers the existence of standard measure zero Kakeya sets in two dimensions. The two-dimensional case is most easily visualized and is worth keeping in mind while reading this section. 

    Let $\Hd_{\gr(d-n,d)}^s$ denote the $s$-dimensional Hausdorff measure on $\gr(d-n, d)$ induced by $d_{\Gr(d-n, d)}$: \[\mathcal{H}_{\gr(d-n,d)}^s(G) = \liminf_{\de \searrow 0}\left\{\sum_{i=1}^{\infty} r_i^s \Bigm\vert \exists \{E_i\}_{i \in \N} \subset \text{Gr}(d-n,d) \text{ s.t. } G \subset \bigcup_{i=1}^\infty E_i, \diam(E_i) = r_i < \de \right\}.\]Let $\mathcal{N} \subset \gr(d-n,d)$ be a fixed closed set. By a covering argument, we may assume for the proof of Theorem~\ref{thm:hkak} that for every $V \in \mathcal{N}$, $d_{\gr(d-n,d)}(V, \text{span}(e_1, \dots, e_{d-n})) \le c_{d,n}$ for a dimensional constant $c_{d,n}$.  

    \begin{definition}
        \begin{itemize}
            \item Define $\mathcal{K}$ to be the set of compact subsets of $[-2,2]^d$, equipped with the Hausdorff metric. 
            \item Define $\mathcal{P}$ to be the set of $E \in \mathcal{K}$ such that for any $V \in \mathcal{N}$, there exists $x \in [-2,2]^d$ such that $x+V \cap [-2,2]^d \subset E$. Note that every element of $\mathcal{P}$ is an $\mathcal{N}$-Kakeya set.
            \item We say sets in $A_{h,\e}$ which can be covered by balls $B_1, \dots, B_k$ of radius $\le \e$ and such that $\sum_{i=1}^k \text{radius}(B_i)^{s + d - n} < C\e^{s + d-n}$ satisfy the \emph{$h, \e$ area condition}. For $h = (h_1, \dots, h_{d-n}) \in [-2,2]^{d-n}$ and $\e > 0$, we define $\mathcal{P}(h, \e)$ to be the set of all $E \in \mathcal{P}$ such that $E \cap A_{h, \e}$ satisfy the $h, \e$ area condition. 
        \end{itemize}
        
    \end{definition}

    The proof of Theorem~\ref{thm:hkak} follows from proving that $\mathcal{P}(h,\e)$ is open and dense. To see this, we apply the Baire category theorem to prove that $\mathcal{P}_0 = \bigcap_{m \in \N}\bigcap_{h \in \Gamma_{1/m}([-2,2]^{d-n})} \mathcal{P}(h, 1/m)$ is nonempty. Take $E \in \mathcal{P}_0$. For any $m \in \N$, $E \in \bigcap_{h \in \Gamma_{1/m}([-2,2]^{d-n})} \mathcal{P}(h, 1/m)$, so for each $h \in \Gamma_{1/m}([-2,2]^{d-n})$, we can cover $E \cap A_{h,\e}$ with balls $B_1^h, \dots, B_k^h$ such that $\text{radius}(B_i^h) < 1/m$ and $\sum_{i=1}^k \text{radius}(B^h_i)^{s+d-n} < m^{-(s+d-n)}$. Then we can cover all of $E$ with balls $\{B_i^h: h \in \Gamma_{1/m}, 1 \le i \le k\}$, all of which have $\text{radius}(B_i^h) < 1/m$ and satisfy \[\sum_{h \in \Gamma_{1/m}}\sum_{i=1}^k \text{radius}(B^h_i)^{s+d-n} < m^{-(s + d-n)}m^{d-n} = m^{-s}.\]Since this holds for any $m \in \N$, we see that $\mathcal{H}^{s+d-n}(E) = 0$.

    \begin{proposition}\label{prop:propk}
        If $d_{\gr(d-n,d)}(\mathcal{N}, \text{span}\{e_1, \dots, e_{d-n}\}) \lesssim 1$ and $\Hd_{\gr(d-n,d)}^{n}(\mathcal{N}) < \infty$, then for any $h \in [-2,2]^{d-n}$ and $\e > 0$, $\mathcal{P}(h,\e)$ is open and dense in $\mathcal{P}$.
    \end{proposition}

    \begin{proof}[Proof of Proposition~\ref{prop:propk}]
        First, we will prove that $\mathcal{P}(h,\e)$ is open. If $E \in \mathcal{P}(h, \e)$ then by the compactness of $E$, there exists a collection of balls $B_1, \dots, B_k$ covering $E \cap A_{h,e}$ such that $\sum_{i=1}^k \text{radius}(B_i)^{s + d - n} < C\e^{s + d-n}$. For some small $\de > 0$, we can increase the radius of each $B_i$ by some $\de$ to $\tilde{B}_i$, while ensuring that $\sum_{i=1}^k \text{radius}(\tilde{B_i})^{s + d - n} < C\e^{s + d-n}$ as well. If $E' \in \mathcal{P}$ and $d_{H}(E, E') < \de$, then $E' \cap A_{h, \e} \subset \bigcup_{i=1}^k \tilde{B}_i$ and hence $E' \cap A_{h, \e}$ satisfies the $h, \e$ area condition. Therefore. $\mathcal{P}(h,\e)$ is open.
        
        Now, let's prove that $\mathcal{P}(h, \e)$ is dense in $\mathcal{P}$. Fix $E \in \mathcal{P}$ and $\de > 0$. We need to find $F \in \mathcal{P}(h,\e)$ such that $d_H(E,F) \lesssim \de$. Recalling the definition of the Hausdorff metric, we need to prove that $\rho(E,F) \lesssim \de$ and $\rho(F,E) \lesssim \de$. Suppose we can find sets $E' \subset E, E_{\mathcal{U}} \subset [-2,2]^d$ with the following properties:
        \begin{enumerate}[label = (\alph*), font = \bfseries]
            \item $\mathcal{H}^{s+d-n}(E') = 0$,
            \item $\rho(E,E') < \de$,
            \item $\rho(E_{\mathcal{U}}, E) \lesssim \de$, and
            \item $E_{\mathcal{U}} \in \mathcal{P}(h,\e)$.
        \end{enumerate}

        If we set $F = E' \cup E_{\mathcal{U}}$, then $F \in \mathcal{P}$ because $E_{\mathcal{U}} \in \mathcal{P}$ and $E'$ is compact. Also, $F \cap A_{h, \e}$ satisfies the $h, \e$ area condition, because $\mathcal{H}^{s+d-n}(E') = 0$ and $E_{\mathcal{U}}$ satisfies the $h, \e$ area condition, so $F \in \mathcal{P}(h,\e)$. Since $E' \subset E$, $d_H(F,E) \le \rho(E_{\mathcal{U}},E) + \rho(E, E') \lesssim \de$. Since $E$ and $\de$ are arbitrary, we can conclude that $\mathcal{P}(h,\e)$ is dense in $\mathcal{P}$.

        So it remains to find the sets $E'$ and $E_{\mathcal{U}}$ and verify the four properties we stated about them. We start with constructing $E'$ and proving it satisfies \textbf{(a)} and \textbf{(b)}. Since $E$ is compact, we can find a collection $x_1, \dots, x_k \in \R^d$, $V_1, \dots, V_k \in \mathcal{N}$ such that $V_i + x_i \in E$ for all $i$ and the $\de$-neighborhood of $V_i + x_i$ covers $E$. Set $E' = \bigcup_{i=1}^k (V_i + x_i)$. By construction, $E' \subset E$. Since $E'$ is a finite union of $\mathcal{H}^{s+d-n}$ null sets, $\mathcal{H}^{s+d-n}(E') = 0$, satisfying \textbf{(a)}. Since the $\de$-neighborhood of $E'$ covers $E$, $\rho(E, E') < \de$, satisfying \textbf{(b)}. 

        Now we define $E_{\mathcal{U}}$. Take $\mathcal{U}$ to be a finite collection of balls $B(V_i, \rho_i)$ covering $\mathcal{N}$ such that $\rho_i < \de$ for all $i$ and $\sum_i \rho_i^{s} \lesssim_{\mathcal{N}} 1$. Since $\mathcal{N}$ is compact and $\Hd_{\gr(d-n,d)}^{n}(\mathcal{N}) < \infty$, we know we can find such a set $\mathcal{U}$. For each plane $V_i$, we can find a point $v_i \in \R^n$ such that $V_i + (h,v_i) \in E$. Each ball $B(V_i, \rho_i)$ induces a set $E_{i}$ given by \[E_{i} = \bigcup \{V' + (h,v_i) : V' \in B_{\rho_i}(V_i)\}.\]Geometrically, we obtain $E_{B_i}$ from $V_i + (h,v_i)$ by rotating $V_i$ in any direction through an angle of at most $\rho_i$. Define $E_{\mathcal{U}} = \bigcup_{B(V_i, \rho_i) \in \mathcal{U}} E_i$.
        
        For any $V, V' \in \gr(d-n,d)$ and any point $x \in [-2,2^d]$, $\rho(V +x,V' + x) \lesssim d_{\gr(d-n,d)}(V,V')$, so since $\mathcal{U}$ consists of balls of radius $<\de$, $\rho(E_{\mathcal{U}}, E) \lesssim \de$, fulfilling \textbf{(c)}.

        To prove \textbf{(d)}, we need to prove that $E_{\mathcal{U}} \in \mathcal{P}$ and $E_{\mathcal{U}} \cap A_{h, \e}$ satisfies the $h, \e$ area condition. Since $\mathcal{N} \subset \bigcup_{B \in \mathcal{U}} B$, we see that $E_{\mathcal{U}} \in \mathcal{P}$. For the $h, \e$ area condition, note that since $E_{\mathcal{U}} \cap A_{h, \e} = \bigcup_{i=1}^k E_i \cap A_{h, \e}$, it suffices to prove that $E_i \cap A_{h,\e}$ can be covered by balls $B^i_1, \dots, B^i_{k_i}$ such that $\sum_{j=1}^{k_i} \text{radius}(B^i_j)^{s+d-n} \lesssim \e^d \rho_i^s$, since $\sum_{i} \rho_i^s \lesssim_{\mathcal{N}} 1$. But $E_i \cap A_{h, \e}$ is contained in a prism of length $\e\rho_i$ in $n$ directions and $\e$ in the remaining $d-n$ directions and hence can be covered by $\rho_i^{n-d}$ many $\e\rho_i$ balls. It follows that \[\sum_{j=1}^{k_i} \text{radius}(B_j^i)^{s+d-n} \le \rho_i^{n-d} (\e \rho_i)^{s + d-n} = \e^{s+d-n}\rho_i^s.\]Hence, \textbf{(d)} holds, completing the proof of the proposition.
    \end{proof}

    \section{Remarks}

    \begin{enumerate}
        \item Theorem \ref{thm:main} concerns the failure of $L^{p,s} \to L^{2d/n, r}$ restriction estimates for an $n$-dimensional quadratic manifold in $\R^d$. For many such manifolds, it already known that the extension operator is not bounded $L^p \to L^q$ for some $q > \frac{2d}{n}$ and any $p$, in which case the Lorentz-space endpoint bound clearly must fail as well. To clarify when our results may be relevant, we briefly discuss when we know or expect extension estimates to fail for some $q > \frac{2d}{n}$.
        
        It is a well-known conjecture that well-curved manifolds of codimension one or two should satisfy the restriction estimates up to $\frac{2d}{n}$. In the codimension one case, this dates back to the original formulation of the restriction problem. The codimension two conjecture was given by Christ in 1982~\cite{C82}. By taking Cartesian products of lower codimensional manifolds, we can find higher codimensional manifolds which should also satisfy restriction estimates up to $\frac{2d}{n}$, but this technique only works to give examples of codimension at most $\frac{d}{2}$. For all manifolds with codimension at most $\frac{d}{2}$, we know curvature is necessary for the restriction estimates up to $\frac{2d}{n}$ to be possible, although we cannot characterize well-curved manifolds except for in the codimension one and two case\footnote{A recent paper by Gressman constructed affine-invariant measures on submanifolds of arbitrary codimension~\cite{G19}. Better understanding these measures could give us more information about when we should expect manifolds to have better or worse restriction estimates.}. It was conjectured by Mockenhaupt in 1996 that restriction estimates up to $\frac{2d}{n}$ were only possible for codimension at most $\frac{d}{2}$ and he proved the conjecture when $d$ is odd~\cite{M96}. The case that $d$ is even is still generally open. 
    
        \item Theorem \ref{thm:hkak} is similar to the main theorem a 2004 paper by Wisewell, which proved that for a family of $n$-dimensional submanifolds $\Gamma(y, \omega) \subset \R^d$ parametrized by $y \in \R^p$, $\omega \in \R^q$ with $p \le d-n \le q$, for each $y$, we can choose $\omega_y$ so that $\bigcup_{y \in \R^p} \Gamma(y, \omega_y)$ has measure zero \cite{W04}\footnote{The author thanks Robert Fraser for suggesting this paper to him.}. Surprisingly, we can even choose the assignment $y \mapsto \omega_y$ independently of the family of the family of manifolds $\Gamma$. This is in many ways more general than the results on packing planes in this paper, but does not appear to yield an estimate on the number of $\de$ balls needed to cover the measure zero set and it does not yield compact sets or apply to Hausdorff measures of smaller dimension than $d$. 
        \item We prove in this paper that the Fourier extension operator for $n$-dimensional quadratic manifolds in $\R^d$ does not satisfy restricted weak-type $(p, 2d/n)$ bounds. As discussed previously, for this rules out restricted weak-type bounds at the restriction endpoint for many quadratic manifolds. One could ask if the same could be said for other classes of manifolds. Our approach does not seem to rule a restriced weak-type endpoint bound for $n$-dimensional quadratic manifolds in $\R^d$ where the restriction endpoint is greater than $2d/n$. If we could find an $\mathcal{N}$-Kakeya set in $\R^d$ with Minkowski dimension $<d$, when $\mathcal{N}$ is the normal space for $\mathcal{M}$, we could use the approach of this paper to disprove restricted weak-type restriction estimates above $2d/n$, but the existence of such sets seems unlikely. 
        
        There are some manifolds, for example, the moment curve, where the extension operator does satisfy weak-type bounds at the restriction endpoint. For most manifolds, however, the problem  of proving weak-type endpoint bounds remains open. Our proof requires certain geometric properties unique to quadratic manifolds or small perturbations thereof, so it would need significant modification to apply to most other manifolds. Specifically, it is important for in the proofs of both Theorem \ref{thm:kmf} and Theorem \ref{thm:hkak} that we can approximate $\mathcal{E}(\chi_{\eta, \de, v})$ with a rectangular prism $T_{\eta, \de, v}$ with only two side lengths, $\de^{-1}$ and $\de^{-2}$ and this is certainly not the case for general manifolds.
    \end{enumerate}

    \printbibliography

\end{document}